\documentclass[leqno]{article}
\usepackage{beamerarticle} 
\usepackage{geometry}      \usepackage{amsthm}        

\usepackage[T2A]{fontenc}
\usepackage[utf8]{inputenc}
\usepackage[english,russian]{babel}

\ifdefined\isenglish\else\def\isenglish{1}\fi

\if\isenglish1
  \AtBeginDocument{\selectlanguage{english}}
\else
  \AtBeginDocument{\selectlanguage{russian}}
\fi

\ifdefined\whichmode\else\def\whichmode{1}\fi

\newcommand{\mytitle}{On the estimation of the Gromov-Hausdorff distance}

\newcommand{\myabstract}{

The paper investigates the Gromov--Hausdorff distance using a new tool --- the so-called $\varkappa$-traces of metric spaces. A general lower bound for the Gromov--Hausdorff distance is obtained in terms of $\varkappa$-traces. Using this bound, exact estimates for the distance between an arbitrary metric space and its Hausdorff hyperspace are derived, and the behavior of the metric on transfinite towers (inductive limits) of hyperspaces is studied. In the final part of the paper, the trace machinery is applied to $l^{p}$-products of metric spaces. A connection is established between computing the Gromov--Hausdorff distance for infinite $l^{\infty }$-products and partition calculus. Using methods from Ramsey theory (in particular, the Erd\"os--Rado theorem), $\varkappa$-traces for $l^{\infty }$-products are computed, and sufficient conditions are obtained under which the Gromov--Hausdorff distance between infinite Cartesian products can be computed exactly. The final section also provides a general estimate for the Gromov–Hausdorff distance between \(l^{\infty }\)-products of metric spaces using the Erd\"os--Rado theorem.

Bibliography: 21 titles.
}

\newcommand{\myname}{E.~Abdullaev}

\newcommand{\mydate}{\ifnum\day<10 0\fi\the\day.\ifnum\month<10 0\fi\the\month.\the\year}

\usepackage{mathtools} \usepackage{amsfonts,amssymb,mathrsfs,amscd}
\usepackage{stmaryrd}

\usepackage{graphicx}
\usepackage{verbatim}
\usepackage{url}
\usepackage{alltt}
\usepackage{indentfirst}
\usepackage{qrcode}

\theoremstyle{plain}

\newtheorem{thm}{Theorem}[section]
\newtheorem{prop}[thm]{Proposition}
\newtheorem{cor}[thm]{Corollary}
\newtheorem{ass}[thm]{Assertion}
\newtheorem{lem}[thm]{Lemma}

\theoremstyle{definition}

\newtheorem{examp}[thm]{Example}

\newtheorem{dfn}[thm]{Definition}
\newtheorem{rk}[thm]{Remark}

\newtheorem*{agree}{Convention}
\newtheorem*{notation}{Notation}
\newtheorem*{quest}{Question}

\newcommand{\e}{\varepsilon}
\newcommand{\vk}{\varkappa}

\newcommand{\cGH}{\mathcal{GH}}

\newcommand{\cH}{\mathcal{H}}
\newcommand{\cR}{\mathcal{R}}
\newcommand{\cM}{\mathcal{M}}

\newcommand{\N}{\mathbb{N}}

\newcommand{\R}{\mathbb{R}}
\newcommand{\Z}{\mathbb{Z}}

\newcommand{\diam}{{\operatorname{diam}}}
\newcommand{\dis}{{\operatorname{dis}}}
\newcommand{\codis}{{\operatorname{codis}}}
\newcommand{\tr}{{\operatorname{tr}}}
\newcommand{\Tr}{{\operatorname{Tr}}}
\newcommand{\dens}{{\operatorname{dens}}}

\renewcommand{\:}{\colon}
\newcommand{\rom}[1]{{\em #1}}

\makeatletter
\@ifundefined{RBibitem}{
\def\RBibitem#1{\normalfont\bibitem{#1}\ignorespaces}

\def\by{\unskip\normalfont~}
  \def\paper{\unskip\normalfont. }
  \def\jour{\unskip\normalfont, \itshape }
  \def\publaddr{\unskip\normalfont, }
  \def\publ{\unskip\normalfont, }
  \def\pages{\unskip\normalfont, pp.~}

\def\inbook{\unskip\normalfont\ \itshape }
  \def\eds{\unskip\normalfont, ed. by }
  \def\book{\unskip\normalfont. \itshape }
  \def\issue{\unskip\normalfont, no.~}

\def\yr{\unskip, \normalfont}
  \def\vol{\unskip, \normalfont\bfseries}

  \def\bysame{\by\hbox to3em{\hrulefill}\thinspace\kern\z@}
  \def\preprint{\unskip\normalfont, preprint }

  \def\finalinfo{\unskip\normalfont, }
  \def\@@finalinfo{\finalinfo}

}{}
\makeatother

\usepackage{hyperref}      
\begin{document}
\title{\mytitle}
\author{\myname}
\date{\today}

\maketitle

\begin{abstract}
\myabstract
\end{abstract}

\ifnum\whichmode=1
\tableofcontents
\fi

\section*{Introduction}

The concept of distance between metric spaces, introduced by F.~Hausdorff \cite{Hausdorff} and generalized in the papers of D.~Edwards~\cite{Edwards} and M.~Gromov \cite{Gromov}, is one of the fundamental tools of modern metric geometry.
This theory was developed further in studies that take into account additional structures on metric spaces and also examine properties of the Gromov--Hausdorff distance itself (see, for example,~\cite{BurBurIva, IvaNikTuz, NP01}).
An important stage in this area was the work of S.~A.~Bogatyy and A.~A.~Tuzhilin~\cite{BogatyTuzhilin}, as well as A.~O.~Ivanov and A.~A.~Tuzhilin~\cite{IvaTuzSimplex, GrigIvaTuz, IvaTuzBorsuk}, devoted to computing the Gromov--Hausdorff distance to simplices and to studying its fundamental properties.

Despite its theoretical importance, the exact computation of the Gromov--Hausdorff distance is an extremely hard algorithmic problem even for finite metric spaces~\cite{NP01, Schmiedl2017}.

The computation and estimation of the Gromov--Hausdorff distance for various classes of metric spaces have been studied in detail in many papers. For instance, in \cite{Adams, LimMemoly} bounds for spheres were obtained, and in \cite{IvaTuzSimplex, GrigIvaTuz} methods were developed that made it possible to compute the distances for a wide class of objects, including simplices and ultrametric spaces. In \cite{IvaTuzBorsuk} sufficient conditions for the attainability of the upper bound are given by means of the Borsuk number (that is, in the case where the Gromov--Hausdorff distance between two metric spaces equals half the maximum of their diameters). Further examples are the computations of the Gromov--Hausdorff distance between a segment and a circle \cite{JiTuzhilin} and between $\Z^2$ and $\R^2$ with the Euclidean metric \cite{MikhAsymptDim}. In~\cite{Me_GH_Product_2026} two-sided bounds were obtained for the Gromov--Hausdorff distance between $l^p$-products of metric spaces.

For this reason, efficiently computable metric invariants that allow one to estimate the Gromov--Hausdorff distance are of particular interest.

In~\cite{Breugel2004} it is proved that for any bounded metric space $X$ its Hausdorff hyperspace $\cH(X)$ is never isometric to the original space $X$, except in the case where $X$ consists of a single point.
However, this non-isometry result leaves open the question of whether the Gromov--Hausdorff distance $d_{GH}(X, \cH(X))$ can vanish.
Note also that the answer to this question is not a trivial consequence of the non-isometry of these spaces. It is known that, on the class of bounded spaces, the equality $d_{GH}(X, Y) = 0$, generally speaking, does not guarantee the isometry of the spaces $X$ and $Y$.
A classical counterexample is the interval $(0, 1)$ and the segment $[0, 1]$, whose Gromov--Hausdorff distance is zero.
Thus, the fact that $d_{GH}(X, \cH(X)) > 0$ requires proof.

In classical metric geometry, packing numbers and metric entropy are widely used to describe the capacity of spaces.
In the present paper we propose an approach dual to packing numbers, namely the concept of $\vk$-traces of a metric space.
In Section 2 we formalize this notion and prove a general lower bound for the Gromov--Hausdorff distance in terms of the difference of $\vk$-traces (Theorem~\ref{thm:main}). This tool enables us to prove the strict positivity of the distance between a wide class of spaces and their Hausdorff hyperspaces.
In Section 3 this technique is extended to transfinite towers of hyperspaces, where we prove that at some (quite specific) level the Gromov--Hausdorff distance inevitably accumulates up to the diameter of the original space (Theorem~\ref{thm:dens_plus1}).
Finally, in Section 4 the trace machinery is applied to the problem of estimating the Gromov--Hausdorff distance between $l^p$-products of metric spaces.
We show that the exact computation of traces for Cartesian products with the $l^\infty$-metric runs into set-theoretic obstructions related to colorings of complete graphs.
Using classical results of Ramsey theory and the Erd\H{o}s--Rado theorem on the partition calculus, we obtain exact values of the $\vk$-traces for infinite products and derive sufficient conditions for the exact computation of the Gromov--Hausdorff distance between them (Theorem~\ref{thm:Calculation_Of_GH_between_products_by_Tr}).
In the very last subsection (4.3) we obtain a general lower bound for the Gromov--Hausdorff distance between $l^{\infty}$-products of metric spaces, using only the Erd\H{o}s--Rado theorem (without using the $\vk$-trace machinery).

\subsection*{Acknowledgments}
\begin{sloppypar}
The author expresses gratitude to his scientific advisor, Doctor of Physical and Mathematical Sciences, Professor \mbox{A.~A.~Tuzhilin}, and Doctor of Physical and Mathematical Sciences, Professor \mbox{A.~O.~Ivanov}, for formulating the problem and their constant attention to the work.
\end{sloppypar}

\section{Main definitions: Gromov--Hausdorff distance}

Let $\cGH$ be the class of all nonempty metric spaces, $\cM \subset \cGH$ the class of all nonempty metric compacta.
In what follows we always consider only nonempty metric spaces.
For arbitrary $X, Y \in \cGH$ let $d_X$ and $d_Y$ denote their metrics.
By a \emph{correspondence} between $X$ and $Y$ we mean a subset $R \subset X \times Y$ whose projections to $X$ and $Y$ are surjective. We denote the set of all correspondences between $X$ and $Y$ by $\cR(X, Y)$.
\emph{The distortion} of a correspondence $R \in \cR(X, Y)$ is defined as the quantity
$$ \dis(R) = \sup \Bigl\{ \bigl|d_X(x, x') - d_Y(y, y')\bigr| \: (x, y), (x', y') \in R \Bigr\}.$$
Then the Gromov--Hausdorff distance is defined as
$$d_{GH}(X, Y) = \frac{1}{2} \inf \bigl\{ \dis(R) \mid R \in \cR(X, Y) \bigr\}. $$
In what follows, when speaking about properties of the Gromov--Hausdorff distance, we follow \cite{BurBurIva}, \cite{BogatyTuzhilin}.
For an arbitrary metric space $(X, d_X)$ we define the \emph{diameter}
$$\diam(X) = \sup_{x, \tilde{x} \in X} d_{X}(x, \tilde{x}).$$
For any bounded metric spaces $X, Y$ the following estimate holds:
$$ \frac{1}{2} \bigl|\diam(X) - \diam(Y)\bigr| \le d_{GH}(X, Y) \le \frac{1}{2} \max\bigl\{\diam(X), \diam(Y)\bigr\}. $$
Let $\{pt\}$ denote the metric space consisting of a single point. Then
for any metric space $X$ the equality holds
$$d_{GH}(X, \{pt\}) = \frac{1}{2} \diam(X).$$
For any $X, Y \in \cGH$ and arbitrary maps $f \: X \to Y$ and $g \: Y \to X$
we set
$$R_{f, g} = \Bigl\{ \bigl(x, f(x)\bigr) \ \big| \ x \in X\Bigr\} \cup \Bigl\{ \bigl(g(y), y\bigr) \ \big| \ y \in Y \Bigr\},$$ then $R_{f, g}\in \cR(X, Y)$.
By the distortion of a map $f \: X \to Y$, where $X, Y \in \cGH$, we mean the distortion of its graph, that is,
$$\dis(f) = \sup_{x, \tilde{x} \in X} \Bigl| d_{X}(x, \tilde{x}) - d_{Y}\bigl(f(x), f(\tilde{x})\bigr)\Bigr|.$$
For a pair of maps $f \: X \to Y$ and $g \: Y \to X$ we introduce their \emph{co-distortion}\/:
$$\codis(f, g) = \sup_{x \in X, \ y \in Y} \Bigl| d_{X}\bigl(x, g(y)\bigr) - d_{Y}\bigl(y, f(x)\bigr)\Bigr|.$$
\begin{prop}[\cite{BogatyTuzhilin}]
For any $X, Y \in \cGH$ and arbitrary maps $f \: X \to Y$ and $g \: Y \to X$
 the following holds
$$\dis(R_{f, g}) = \max\bigl\{\dis(f), \dis(g), \codis(f, g)\bigr\}.$$
\end{prop}
We shall also need the following equivalent definition, which follows from the preceding proposition:
\begin{prop}\label{prop:def}
For any $X, Y \in \cGH$ the following holds
$$d_{GH}(X, Y) = \frac{1}{2} \inf_{\substack{f \: X \to Y \\ g \: Y \to X}} \max \bigl\{ \dis(f), \dis(g), \codis(f, g)\bigr\},$$
where the infimum is taken over all maps $f$ and $g$.
\end{prop}

\section{Traces of metric spaces}

In metric geometry, to estimate the ``capacity'' of a metric space $X$ one uses packing numbers, equal to the maximum cardinality of an $r$-separated subset in $X$.
In this paper we propose a dual approach: we fix the cardinality of the subset (the cardinal $\vk$) and look for the maximum distance $r$, by which one can ``spread apart'' $\vk$ points inside $X$.
This quantity, which we call the $\vk$-trace, is an invariant of the metric space and, as will be shown below, estimates the Gromov--Hausdorff distance in a natural way.

\begin{notation}
For any set $X$, let $\# X$ denote its cardinality.
\end{notation}

\subsection{Definitions and properties}

\begin{dfn}
A subset $Y \subset X$ of a metric space $(X, d_{X})$ is called $r$-separated if for any two points $x \ne y \in Y$ the inequality $d_{X}(x, y) \ge r$.
\end{dfn}

\begin{dfn}
Let $(X, d)$ be a metric space and $\vk$ be an arbitrary cardinal, with $\vk \ge 2$.
We call the $\vk$-\emph{trace} of $X$ the number
$$
\Tr_X(\vk) = \sup \bigl\{ r > 0 : \text{there exists an $r$-separated subset } Q_{r} \subset X \text{ of cardinality } \#Q_{r} = \vk \bigr\}.
$$
By definition, we assume that if no $r$-separated subsets of the given cardinality exist, then $\Tr_{X}(\vk) = 0$.
We also assume by definition that the supremum of a set unbounded above equals infinity.
\end{dfn}

\begin{dfn}
Let $(X, d)$ be a metric space and $\vk$ an arbitrary cardinal, with $\vk \ge 2$.
We call the (\emph{bounded}) $\vk$-\emph{trace} of $X$ the number
$$
\tr_{X}(\vk) = \sup \Bigl\{
r > 0 \;\Big|\;
\begin{array}{l}
\text{there exists a bounded } r\text{-separated} \\
\text{subset } Q_{r} \subset X \text{ of cardinality } \#Q_{r} = \vk
\end{array}
\Bigr\}.
$$
\end{dfn}

\begin{examp}
Let $\Delta_n(a)$ be the discrete metric space of $n$ points (where $n$ is an arbitrary cardinal) whose pairwise distances equal $a > 0$. Then the definitions show that for every cardinal $\vk \ge 2$ the following holds:
$$ \tr_{\Delta_n(a)}(\vk) =
\begin{cases}
a, & \text{if } \vk \le n, \\
0, & \text{if } \vk > n.
\end{cases} $$
\end{examp}

\begin{examp}[Trace of a segment]\label{examp:01}
Consider $X = [0, 1]$ with the natural metric $d_{X}(x, y) = |x - y|$.

For every positive integer $n \ge 2$ we obtain
$$\tr_{X}(n) = \frac{1}{n - 1},$$
where the ($1/(n-1)$)-separated sets from the definition of the bounded trace are to be taken as
$$Q_{1/(n-1)} = \biggl\{ 0, \ \frac{1}{n - 1}, \ \frac{2}{n - 1}, \ \dots, \ 1 \biggr\},$$
while for every infinite cardinal $\vk$ we have $\tr_{X}(\vk) = 0$.
\end{examp}

\begin{examp}[Trace of a circle] \label{examp:S1}
Consider the circle with the intrinsic metric:
$$X = S^{1} := \bigl\{ (x, y) \in \R^{2} \mid x^{2} + y^{2} = 1\bigr\}.$$

For every positive integer $n \ge 2$ we obtain
$$\tr_{X}(n) = \frac{2\pi}{n},$$
where the $2\pi/n$-separated sets from the definition of the bounded trace are to be taken as the vertices of a regular $n$-gon
$$Q_{2\pi/n} = \Biggl\{ \biggl( \cos\Bigl(\frac{2 \pi k}{n}\Bigr), \ \sin\Bigl(\frac{2 \pi k}{n}\Bigr) \biggr) \ \bigg| \ k \in \{0, \dots, n - 1\} \Biggr\}.$$
\end{examp}

\begin{notation}
For any metric space $(X, d_{X})$ and any $\lambda > 0$, let $\lambda X$ denote the metric space $(X, d_{\lambda X})$, in which, for any $x, y \in X$, by definition
$$d_{\lambda X}(x, y) = \lambda d_{X}(x, y).$$
\end{notation}

\begin{dfn}[\cite{Engelking}]
The density of a topological space $X$ is the least cardinality of dense subsets of $X$; it is denoted by $\dens(X)$.
\end{dfn}

Note that the density of a space does not exceed its cardinality: $\dens(X) \le \# X$.

The following properties follow immediately from these definitions.

\begin{prop}\label{prop:properties}
Let $X$ be a metric space, and let $\vk \ge 2$ be a cardinal.
Then the following holds.
\begin{enumerate}
\item \label{prop:properties:1}
We have $0 \le \tr_{X}(\vk) \le \Tr_{X}(\vk) \le \diam(X)$.
\item \label{prop:properties:2}
If $\vk$ is finite, then $\tr_{X}(\vk) = \Tr_{X}(\vk)$.
\item \label{prop:properties:3}
In the case $\vk = 2$ the trace coincides with the diameter of the space:
$$\tr_{X}(2) = \Tr_{X}(2) = \diam(X).$$
\item \label{prop:properties:4}
If $X$ is a bounded metric space, then
$\tr_{X}(\vk) = \Tr_{X}(\vk)$.
\item \label{prop:properties:5}
Let $2 \le \vk_1 \le \vk_2$ be cardinals. Then
$\tr_{X}(\vk_{1}) \ge \tr_{X}(\vk_{2})$ and $\Tr_{X}(\vk_{1}) \ge \Tr_{X}(\vk_{2})$.
\item \label{prop:properties:6}
Let $A \subset X$ be an arbitrary subspace.
Then for every cardinal $\vk \ge 2$ we have
$\tr_{A}(\vk) \le \tr_{X}(\vk)$ and $\Tr_{A}(\vk) \le \Tr_{X}(\vk)$.
\item \label{prop:properties:7}
For every $\lambda > 0$ we have $\tr_{\lambda X}(\vk) = \lambda \tr_{X}(\vk)$ and $\Tr_{\lambda X}(\vk) = \lambda \Tr_{X}(\vk)$.
\item \label{prop:properties:8}
Let $A \subset X$ be a dense subset. Then
$\tr_{A}(\vk) = \tr_{X}(\vk)$ and $\Tr_{A}(\vk) = \Tr_{X}(\vk)$.
\item \label{prop:properties:9}
If $\vk > \dens(X)$, then $\tr_{X}(\vk) = \Tr_{X}(\vk) = 0$.
\end{enumerate}
\end{prop}

\begin{proof}
Items~\ref{prop:properties:1}--\ref{prop:properties:7} follow immediately from the definitions. We prove items~\ref{prop:properties:8} and~\ref{prop:properties:9}.

~(\ref{prop:properties:8})
The inequality $\tr_A(\vk) \le \tr_X(\vk)$ follows from item~\ref{prop:properties:6}.
We prove the reverse inequality.
Let $r < \tr_X(\vk)$.
There exists a bounded $r$-separated set $Q \subset X$ of cardinality $\vk$.
Choose an arbitrary $\e < r/2$.
Since the set $A$ is dense in $X$, for every point $q \in Q$ there is a point $a_q \in A$ such that $d_X(q, a_q) < \e$.
For any distinct $q, q' \in Q$, by the triangle inequality we have:
$$d_X(a_q, a_{q'}) \ge d_X(q, q') - d_X(q, a_q) - d_X(q', a_{q'}) > r - 2\e.$$
Consequently, the map $q \mapsto a_q$ is injective, and the set $\{a_q\}_{q \in Q} \subset A$ has cardinality $\vk$, is bounded and $(r-2\e)$-separated.
Hence $\tr_A(\vk) \ge r - 2\e$.
Letting $\e \to 0$ and $r \to \tr_X(\vk)$, we obtain $\tr_A(\vk) \ge \tr_X(\vk)$.
The case of the trace $\Tr$ is treated similarly.

~(\ref{prop:properties:9})
Suppose that $\vk > \dens(X)$, but $\tr_X(\vk) > 0$.
Then there exists $r > 0$ such that in $X$ there is an $r$-separated set $Q$ of cardinality $\vk$.
Consider the family of open balls $\{B(q, r/3)\}_{q \in Q}$.
Because of the $r$-separability of the set $Q$, these balls are pairwise disjoint. Since every dense set must meet every nonempty open set, the cardinality of any dense set cannot be less than the cardinality of a family of disjoint balls, that is, not less than $\#Q = \vk$.
Consequently, $\dens(X) \ge \vk$, which contradicts the assumption.
Hence $\tr_X(\vk) = 0$.
The same argument applies to the trace $\Tr$.
\end{proof}

\subsection{General estimate of the Gromov--Hausdorff distance}

\begin{dfn}
Let $X, Y$ be metric spaces, $f \: X \to Y$ an arbitrary map. Then $f$ is called bounded if it takes bounded sets to bounded sets.
\end{dfn}

\begin{ass}\label{ass:GH_dfn}
For any $X, Y \in \cGH$ we have
$$d_{GH}(X, Y) = \frac{1}{2} \inf_{\substack{f \: X \to Y \\ g \: Y \to X}} \max \bigl\{ \dis(f), \dis(g), \codis(f, g)\bigr\},$$
where the infimum is taken over all bounded maps $f$ and $g$.
\end{ass}

\begin{proof}
Indeed, if $d_{GH}(X, Y) = \infty$, there is nothing to prove.
Therefore, without loss of generality, we assume that $d_{GH}(X, Y) < \infty$.
In this case the assertion follows immediately from Proposition~\ref{prop:def} and from the fact that any map $f$, for which $\dis(f) < \infty$, is bounded.
\end{proof}

\begin{agree}
Here and below the difference $\bigl| a - b \bigr|$ is understood in the extended sense: if exactly one of the values is infinite, the absolute value of the difference is set equal to $\infty$, and if both values are infinite, their difference is considered equal to $0$.
We shall also consider the difference $a - b$, where $a \ge b$, and likewise understand it in the extended sense: if $a = \infty$ and $b < \infty$, then $a - b := \infty$, and if $a = b = \infty$, then $a - b := 0$.
\end{agree}

\begin{thm}\label{thm:main}
Let $X, Y$ be metric spaces. Then
$$2d_{GH}(X, Y) \ge \sup_{\vk \ge 2} \bigl| \tr_{X}(\vk) - \tr_{Y}(\vk) \bigr|,$$
where the supremum is taken over all cardinals $\vk \ge 2$.
\end{thm}

\begin{proof}
Without loss of generality we assume that $d_{GH}(X, Y) < \infty$, since otherwise there is nothing to prove.
Consider an arbitrary cardinal $\vk \ge 2$.
It suffices to prove that $2d_{GH}(X, Y) \ge \bigl| \tr_{X}(\vk) - \tr_{Y}(\vk) \bigr|$.
To this end, without loss of generality, we may assume that $\tr_{X}(\vk) < \tr_{Y}(\vk)$ (the case $\tr_{X}(\vk) = \tr_{Y}(\vk)$ is trivial). We consider the following cases.
\begin{enumerate}
\item \label{thm:main:proof1}
Let $0 \le \tr_{X}(\vk) < \tr_{Y}(\vk) < \infty$.

For any $\e > 0$ less than $\bigl(\tr_{Y}(\vk) - \tr_{X}(\vk)) / 3$ and not exceeding $\tr_{X}(\vk)$, consider a bounded $\bigl(\tr_{Y}(\vk) - \e \bigr)$-separated set $Q_{Y, \e} \subset Y$ of cardinality $\#Q_{Y, \e} = \vk$.

Consider an arbitrary bounded map $g \: Y \to X$. Then
$$
\dis(g) = \sup_{y, \tilde{y} \in Y} \Bigl| d_{Y}(y, \tilde{y}) - d_{X}\bigl(g(y), g(\tilde{y})\bigr) \Bigr| \ge
\sup_{y, \tilde{y} \in Q_{Y, \e}} \Bigl| d_{Y}(y, \tilde{y}) - d_{X}\bigl(g(y), g(\tilde{y})\bigr) \Bigr|.
$$
If the restriction $g|_{Q_{Y, \e}} \: Q_{Y, \e} \to g(Q_{Y, \e})$ is not injective, then $\dis(g) \ge \tr_{Y}(\vk) - \e \ge \bigl| \tr_{X}(\vk) - \tr_{Y}(\vk)\bigr|$ by virtue of $\e \le \tr_{X}(\vk)$, while for $\tr_{X}(\vk) = 0$ the inequality $\dis(g) \ge \tr_{Y}(\vk)$ is obtained by passing to the limit as $\e \to 0$; in both cases there is nothing to prove.

Further, suppose that $g|_{Q_{Y, \e}} \: Q_{Y, \e} \to g(Q_{Y, \e})$ is injective.
Then $\#g(Q_{Y, \e}) = \vk$; consequently, the set $g(Q_{Y, \e})$ cannot simultaneously be $\bigl(\tr_{X}(\vk) + \e \bigr)$-separated (this set is bounded because the map $g$ is bounded). The latter means that there exist $y, \tilde{y} \in Q_{Y, \e}$, for which $d_{X}(g(y), g(\tilde{y})) < \tr_{X}(\vk) + \e$.
At the same time, for these $y, \tilde{y} \in Q_{Y, \e}$ we have $d_{Y}(y, \tilde{y}) \ge \tr_{Y}(\vk) - \e$.
Since $\e < \bigl(\tr_{Y}(\vk) - \tr_{X}(\vk)\bigr) / 3$, the following inequality holds
$$
\Bigl| d_{Y}(y, \tilde{y}) - d_{X}(g(y), g(\tilde{y})) \Bigr| > \bigl|\tr_{Y}(\vk) - \tr_{X}(\vk) \bigr| - 2\e.
$$
Letting $\e \to 0$, we obtain $\dis(g) \ge \bigl| \tr_{X}(\vk) - \tr_{Y}(\vk) \bigr|$.
By \ref{ass:GH_dfn} we have $2d_{GH}(X, Y) \ge \dis(g) \ge \bigl| \tr_{X}(\vk) - \tr_{Y}(\vk) \bigr|$.

\item Let $\tr_{Y}(\vk) = \infty$. We argue in exactly the same way as in the previous item. If $\tr_{X}(\vk) = \infty$, there is nothing to prove, so we assume that $\tr_{X}(\vk) < \infty$.
Fix a number $M > \tr_{X}(\vk) + 10$ and consider an $M$-separated set $Q_M \subset Y$ of cardinality $\# Q_{M} = \vk$.
For an arbitrary bounded map $g \: Y \to X$ two cases are possible.
\begin{itemize}
\item
If $g|_{Q_M}$ is not injective, then $\dis(g) \ge M$.
\item If $g|_{Q_M}$ is injective, then $\dis(g) \ge |M - \tr_{X}(\vk) - 1|$.
Indeed, $\# g(Q_M) = \vk$ and, consequently, the set $g(Q_M)$ cannot be $(\tr_{X}(\vk) + 1)$-separated.
\end{itemize}
Letting $M \to \infty$, we obtain $d_{GH}(X, Y) = \infty$.
\end{enumerate}
\end{proof}

\begin{rk}
Note that for any $X, Y \in \cGH$ we can consider the distance
$$D(X, Y) = \sup_{\vk} \bigl| \tr_{X}(\vk) - \tr_{Y}(\vk) \bigr|.$$
It is symmetric and satisfies the triangle inequality, that is, it is a generalized pseudometric.
Theorem~\ref{thm:main} asserts that
$$2d_{GH}(X, Y) \ge D(X, Y).$$
\end{rk}

Similarly to Theorem~\ref{thm:main} one can prove the following theorem.

\begin{thm}\label{thm:main2}
Let $X, Y$ be metric spaces. Then
$$2d_{GH}(X, Y) \ge \sup_{\vk \ge 2} \bigl| \Tr_{X}(\vk) - \Tr_{Y}(\vk) \bigr|,$$
where the supremum is taken over all cardinals $\vk \ge 2$.
\end{thm}

\begin{examp}
Using Theorem~\ref{thm:main}, let us estimate the Gromov--Hausdorff distance between the segment $I_{\lambda} = [0, \lambda]$ and the circle
$$S^1 = \bigl\{ (x, y) \in \R^{2} \mid x^{2} + y^{2} = 1\bigr\}$$
with the intrinsic metric.
In Example~\ref{examp:01} it was shown that for $n \in \N$ we have
$\tr_{I_{1}}(n) = \frac{1}{n - 1}$, whence we obtain
$\tr_{I_{\lambda}}(n) = \frac{\lambda}{n - 1}$.
In Example~\ref{examp:S1} it was shown that
$\tr_{S^1}(n) = \frac{2\pi}{n}$.
Then, by Theorem~\ref{thm:main}, we obtain
$$d_{GH}(I_{\lambda}, S^{1}) \ge \frac{1}{2} \max\Bigl\{ \bigl| \diam(I_{\lambda}) - \diam(S^1) \bigr|, \ \bigl|\tr_{I_{\lambda}}(3) - \tr_{S^1}(3)\bigr|\Bigr\} = \frac{1}{2} \max\biggl\{ \Bigl| \lambda - \pi \Bigr|, \ \Bigl|\frac{\lambda}{4} - \frac{\pi}{3} \Bigr| \biggr\}.$$
This distance was computed explicitly in \cite{JiTuzhilin}:
$$
d_{GH}(I_{\lambda},S^1)=
\begin{cases}
\dfrac{\pi}{2} - \dfrac{\lambda}{4}, & \text{ for } \ 0 \le \lambda \le \dfrac{2\pi}{3}, \\[10pt]
\dfrac{\pi}{3}, & \text{ for } \ \dfrac{2\pi}{3} \le \lambda \le \dfrac{5\pi}{3}, \\[10pt]
\dfrac{\lambda - \pi}{2}, & \text{ for } \ \lambda \ge \dfrac{5\pi}{3}.
\end{cases}
$$
Our bound is not sharp in this example (except for the case where the distance equals the difference of the diameters (the case $\lambda \ge 5\pi/3$), which corresponds to the bound for the 2-trace).
\end{examp}

\subsection{Estimate of the Gromov--Hausdorff distance between a metric space and its Hausdorff hyperspace}

\begin{dfn}
Let $(X, d_{X})$ be a generalized pseudometric space.
For nonempty $A, B \subset X$ we define the \emph{Hausdorff distance} between $A$ and $B$:
$$d^{H}_{X}(A, B) = \max \biggl\{ \sup_{a \in A} \inf_{b \in B} d_{X}(a,b), \
\sup_{b \in B} \inf_{a \in A} d_{X}(a, b) \biggr\}.$$
\end{dfn}

\begin{notation}
For an arbitrary metric space $X$ let $\cH(X)$ denote the family of all nonempty closed bounded subsets of $X$.
\end{notation}

\begin{prop}[\cite{BurBurIva}] \label{prop:H}
For every metric space $X$ the Hausdorff distance is a metric on $\cH(X)$.
Moreover, the spaces $X$ and $\cH(X)$ simultaneously possess any one of the following properties\/\rom: completeness, boundedness, total boundedness, compactness.
\end{prop}

\begin{notation}
Let $\vk \ge 2$ be a cardinal. We denote by $2^{\vk} - 1$ the cardinal defined as follows:
$$
2^{\vk} - 1 =
\begin{cases}
2^{\vk} - 1, & \text{if } \vk \text{ is finite}, \\
2^{\vk},     & \text{if } \vk \text{ is infinite}.
\end{cases}
$$
Here in the finite case the right-hand side is the ordinary difference of cardinals, whereas subtracting one from an infinite cardinal is undefined, so we put $2^{\vk} - 1 = 2^{\vk}$ by definition.
\end{notation}

\begin{prop}\label{prop:H_trace}
Let $X$ be a metric space and $\vk \ge 2$ be a cardinal.
Then $\tr_{\cH(X)}(2^{\vk} - 1) \ge \tr_{X}(\vk)$.
\end{prop}

\begin{proof}
Let $Q \subset X$ be a bounded $r$-separated subset of cardinality $\# Q = \vk$.
Then consider the set $\widetilde{Q} = 2^{Q} \setminus \{ \emptyset \} \subset \cH(X)$.
Then $\widetilde{Q} \subset \cH(X)$ is a bounded $r$-separated subset of cardinality $2^{\vk} - 1$.
Passing to the supremum over $r > 0$, we obtain the required result.
\end{proof}

\begin{thm}\label{thm:H_trace_2}
Let $X$ be a metric space and $\vk \ge 2$ be a cardinal.
Then $$2d_{GH}\bigl(X, \cH(X) \bigr) \ge \tr_{X}(\vk) - \tr_{X}(2^{\vk} - 1) \ge 0.$$
In particular,
$2d_{GH}\bigl(X, \cH(X)\bigr) \ge \diam(X) - \tr_{X}(3) \ge 0$.
\end{thm}

\begin{proof}
According to item~\ref{prop:properties:5} of Proposition~\ref{prop:properties}, as well as to Proposition~\ref{prop:H_trace}, we have
$$\tr_{X}(2^{\vk} - 1) \le \tr_{X}(\vk) \le \tr_{\cH(X)}(2^{\vk} - 1).$$
Then, by Theorem~\ref{thm:main} we obtain
\begin{multline*}
2d_{GH}\bigl(X, \cH(X)\bigr) \ge \bigl| \tr_{X}(2^{\vk} - 1) - \tr_{\cH(X)}(2^{\vk} - 1)\bigr| \\ = \tr_{\cH(X)}(2^{\vk} - 1) - \tr_{X}(2^{\vk} - 1) \ge \tr_{X}(\vk) - \tr_{X}(2^{\vk} - 1) \ge 0.
\end{multline*}
\end{proof}

\begin{examp}
Consider the segment $X = [0, 1]$.
Then from Theorem~\ref{thm:H_trace_2} and from Example~\ref{examp:01} we obtain
$$d_{GH}\bigl(X, \cH(X)\bigr) \ge \frac{\diam(X) - \tr_{X}(3)}{2} = \frac{1 - 1/2}{2} = \frac{1}{4}.$$
\end{examp}

Recall that a subset $A$ of a metric space $X$ is called an $\e$-net if for any $x \in X$ there exists a point $a \in A$, satisfying the condition $d_{X}(a, x) \le \e$.
Recall also that a metric space $X$ is called totally bounded if for every $\e > 0$ there exists a finite $\e$-net.

We shall need the following assertion.

\begin{ass}[\cite{BurBurIva}] \label{ass:e_net}
Let $X$ be a metric space.
If there exists an $\e / 3$-net of $n$ points, then any $\e$-separated set $Q \subset X$
cannot contain more than $n$ points.
\end{ass}

\begin{prop}\label{prop:tr_iff}
Let $X$ be a metric space. Then $X$ is totally bounded if and only if $\Tr_{X}(\aleph_{0}) = 0$.
\end{prop}

\begin{proof}
Let $X$ be totally bounded.
Suppose that $r_{0} := \Tr_{X}(\aleph_{0}) > 0$.
Then for every $0 < r < r_{0}$ there exists an $r$-separated set $Q \subset X$ of cardinality $\# Q = \aleph_{0}$.
Since $X$ is totally bounded, there exists a finite $(r / 3)$-net, which contradicts Assertion~\ref{ass:e_net}.

Let $\Tr_{X}(\aleph_{0}) = 0$.
Suppose that $X$ is not totally bounded.
Consider $\e > 0$, for which there are no finite $\e$-nets in $X$.
We construct a countable $\e / 2$-separated set $Q \subset X$ as follows.
Choose a point $x_{1} \in X$. Since the set $\{x_{1}\}$ is not an $\e$-net, there exists a point $x_{2} \in X$ such that $d_{X}(x_{1}, x_{2}) > \e \ge \e / 2$.
Next, the set $\{x_{1}, x_{2}\}$ is not an $\e$-net, so there exists a point $x_{3}$ such that $d_{X}(x_{j}, x_{3}) > \e > \e/2$ for $j = 1, 2$. Continuing the process, we obtain a countable $\e/2$-separated set $Q = \{x_{1}, x_{2}, x_{3}, \dots \}$.
Then $\Tr_{X}(\aleph_{0}) \ge \e/2$, which contradicts our assumption.
\end{proof}

\begin{cor}\label{cor:tr_iff}
Let $X$ be a bounded metric space. Then $X$ is totally bounded if and only if $\tr_{X}(\aleph_{0}) = 0$.
\end{cor}

\begin{proof}
This follows immediately from Proposition~\ref{prop:tr_iff} and from item~\ref{prop:properties:4} of Proposition~\ref{prop:properties}.
\end{proof}

\begin{examp}
Note that in Corollary~\ref{cor:tr_iff} one cannot drop the condition that $X$ be bounded.
Consider $X = \Z, \ d_{X}(x, y) = |x - y|$.
Then $X$ is unbounded, but $\tr_{X}(\aleph_{0}) = 0$, since any countable subset $Q \subset X$ is unbounded.
\end{examp}

\begin{notation}
To describe the cardinality of metric spaces, we introduce the cardinals
$\{\beth_{\alpha}\}$, defined recursively:
$\beth_{0} = \aleph_{0}$~--- the countable cardinality, and $\beth_{\alpha + 1} = 2^{\beth_{\alpha}}$
for every ordinal $\alpha$. For a limit ordinal $\alpha$ we define
$\beth_{\alpha} = \sup_{\beta < \alpha} \beth_{\beta}$.
We denote by $\omega$ the first infinite ordinal; then $\beth_{\omega} = \sup_{n \in \N} \beth_{n}$.
\end{notation}

\begin{cor}\label{thm:H_trace}
Let $X$ be a bounded metric space different from the one-point space $\{pt\}$.
Suppose that $\# X < \beth_{\omega}$.
Then
$$d_{GH}\bigl(X, \cH(X) \bigr) > 0.$$
\end{cor}

\begin{proof}
The first inequality follows from Theorem~\ref{thm:H_trace_2} by passing to the supremum.
It remains to justify the last inequality.
Consider the following family of cardinals:
$$S = \bigl\{\vk \ge 2 \mid \tr_{X}(\vk) \ne \tr_{X}(2^\vk - 1)\bigr\}.$$
We have to show that $S \ne \emptyset$.
We argue by contradiction and consider the following two cases.

\begin{enumerate}
\item
Suppose that $X$ is totally bounded.
This condition is equivalent to the following: for every $\e > 0$ there exists a finite $\e$-net.
Suppose that $S = \emptyset$.
Since in this case for all $n \in \N$ we have $\tr_{X}(n) = \tr_{X}(2^n - 1)$, taking into account item~\ref{prop:properties:5} of Proposition~\ref{prop:properties}, we obtain
$\tr_{X}(n) = \tr_{X}(2) = \diam(X) > 0$ for all $n \in \N$.
Choose a number $\e$ such that $0 < \e < \diam(X) / 10$ and consider a finite $\e$-net $A \subset X$.
Choose a number $n$ such that $n > \# A$.
Since $\tr_{X}(n) = \diam(X) > 0$, there exists a bounded $(\diam(X) - \e)$-separated subset $Q \subset X$ of cardinality $\# Q = n$, which contradicts Assertion~\ref{ass:e_net}.
Thus $S \ne \emptyset$.
\item
Suppose that $X$ is not totally bounded.
Then, by Corollary~\ref{cor:tr_iff} we obtain
$\tr_{X}(\aleph_{0}) > 0$.
Since $S = \emptyset$, for every $n \in \N$ we obtain
$\tr_{X}(\beth_{n}) = \tr_{X}(\aleph_{0}) > 0$.
By the hypothesis of the corollary we must have $\# X < \beth_{\omega}$, which means that there exists a number $N \in \N$ for which $\# X < \beth_{N}$.
By item~\ref{prop:properties:9} of Proposition~\ref{prop:properties} we obtain $\tr_{X}(\beth_{N}) = 0$, which contradicts the inequality $\tr_{X}(\beth_{N}) > 0$.
Hence $S \ne \emptyset$.
\end{enumerate}
Thus $S \ne \emptyset$, which completes the proof of the corollary.
\end{proof}

\begin{quest}
Can one, in Corollary~\ref{thm:H_trace}, drop the condition $\# X < \beth_{\omega}$; that is, is it true for every bounded metric space $X \ne \{pt\}$ that
$$d_{GH}\bigl(X, \cH(X)\bigr) > 0?$$
\end{quest}

\section{Towers of metric spaces}

\subsection{Definitions and properties}

Consider the operation of passing to the Hausdorff hyperspace $\cH(X)$. Each
point $x \in X$ can be naturally identified with a one-point
subset $\{x\} \in \cH(X)$. Since the map $X \to \cH(X), \ x \mapsto \{x\}$
 is an isometric embedding, we
can regard $X$ as a subspace of $\cH(X)$. Iterating this
 process transfinitely, we arrive at the notion of a tower.

\begin{dfn}
Let $\alpha$ be an arbitrary ordinal. \emph{A tower of metric spaces of length $\alpha$}
(or an $\alpha$-tower) is a family of metric spaces $\{(X_{\beta}, d_{X_\beta})\}_{\beta < \alpha}$,
 satisfying the following conditions:
\begin{enumerate}
    \item if $\beta < \gamma < \alpha$, then $X_{\beta} \subset X_{\gamma}$, and
    the metric $d_{X_\gamma}$ on the subspace $X_{\beta}$ coincides with $d_{X_\beta}$;
    \item if $\gamma < \alpha$ is a limit ordinal, then the space $X_{\gamma}$
     is the union of all preceding levels of the tower:
    $$X_{\gamma} = \bigcup_{\beta < \gamma} X_{\beta}.$$
\end{enumerate}
The metric space $(X, d_{X})$, where $X = \bigcup_{\beta < \alpha} X_{\beta}$, and
the metric $d_{X}$ is induced by the levels of the tower (i.e. $d_{X}(x,y) = d_{X_\gamma}(x,y)$ for every
$\gamma < \alpha$, such that $x,y \in X_{\gamma}$), is called the \emph{inductive
limit} (or union limit) of this tower and is denoted by $X = \varinjlim X_{\beta}$.
\end{dfn}

 To study the Gromov--Hausdorff distance between inductive limits, we need
 to carry over the notion of correspondence to the case of towers.

\begin{dfn}
Let $\{(X_{\beta}, d_{X_{\beta}})\}_{\beta < \alpha}$ and $\{(Y_{\beta}, d_{Y_{\beta}})\}_{\beta < \alpha}$ be
two $\alpha$-towers of metric spaces. A family of sets $\{R_{\beta}\}_{\beta < \alpha}$,
where $\emptyset \ne R_{\beta} \subset X_{\beta} \times Y_{\beta}$, is called a \emph{tower of correspondences},
 if:
\begin{enumerate}
    \item for every $\beta < \alpha$ we have $R_{\beta} \in \cR(X_{\beta}, Y_{\beta})$;
    \item if $\beta < \gamma < \alpha$, then $R_{\beta} \subset R_{\gamma}$;
\end{enumerate}
The union $R = \bigcup_{\beta < \alpha} R_{\beta}$ is called the \emph{inductive limit
of the tower of correspondences} (denoted by $R = \varinjlim R_{\beta}$) and represents
a correspondence between the inductive limits $X = \varinjlim X_{\beta}$ and $Y = \varinjlim Y_{\beta}$.
\end{dfn}

\begin{lem}\label{lem:R_varinjlim}
Let $R = \varinjlim R_{\beta}$ be the inductive limit of an $\alpha$-tower of correspondences $\{R_{\beta}\}_{\beta < \alpha}$ between
$\alpha$-towers of metric spaces $\{X_{\beta}\}_{\beta < \alpha}$ and $\{Y_{\beta}\}_{\beta < \alpha}$. Then the distortion of the limit
correspondence equals the supremum of the distortions over the levels of the tower\/\rom:
$$\dis(R) = \sup_{\beta < \alpha} \dis(R_{\beta}).$$
\end{lem}

\begin{proof}
Set
$$X = \bigcup_{\beta < \alpha} X_{\beta}, \quad Y = \bigcup_{\beta < \alpha} Y_{\beta}.$$
Consider $(x, y), (x', y') \in R$.
For them there exist $\beta_{0}, \beta'_{0} < \alpha$ such that $(x, y) \in R_{\beta_{0}}$ and $(x', y') \in R_{\beta'_{0}}$.
Let $\beta = \max\{\beta_{0}, \beta'_{0}\}$.
Then $(x, y), (x', y') \in R_{\beta}$, whence we obtain
$\bigl|d_{X}(x, x') - d_{Y}(y, y')\bigr| \le \dis(R_{\beta}) \le \sup_{\beta < \alpha} \dis(R_{\beta})$.
Passing to the supremum over $(x, y), (x', y') \in R$, we obtain
$$\dis(R) = \sup_{(x, y), (x', y') \in R} \bigl|d_{X}(x, x') - d_{Y}(y, y')\bigr| \le \sup_{\beta < \alpha} \dis(R_{\beta}).$$
The inequality $\dis(R_{\beta}) \le \dis(R)$, where $\beta < \alpha$, follows from the inclusion $R_{\beta} \subset R$. Consequently, $\sup_{\beta < \alpha} \dis(R_{\beta}) \le \dis(R)$ and, thereby, the equality is proved.
\end{proof}

\subsection{Hausdorff tower}

\begin{dfn}
Let $(X, d_X)$ be an arbitrary metric space.
By definition we set $\cH^{0}(X) = X$.
For every ordinal $\gamma \ge 1$ we define the space $\cH^{\gamma}(X)$ by induction:
\begin{itemize}
    \item if $\gamma = \beta + 1$, then $\cH^{\beta + 1}(X) := \cH\bigl(\cH^{\beta}(X)\bigr)$;
    \item if $\gamma$ is a limit ordinal, then $\cH^{\gamma}(X) := \bigcup_{\beta < \gamma} \cH^{\beta}(X)$.
\end{itemize}
The family $\{\cH^{\beta}(X)\}_{\beta < \alpha}$ forms the $\alpha$-\emph{Hausdorff tower},
 given by the space $X$.
\end{dfn}

Note that for every $X \in \cGH$ and every ordinal $\alpha$ we have $\diam\bigl(\cH^{\alpha}(X)\bigr) = \diam(X)$.

\begin{thm}[\cite{Me_d_{GH}_2026}] \label{thm:cH}
For any $X, Y \in \cGH$ we have
$$d_{GH}\bigl(\cH(X), \cH(Y) \bigr) \le d_{GH}(X, Y).$$
Moreover, for any bounded maps $f_{1} \: X \to Y, \ f_{2} \: Y \to X$ consider the bounded maps
$$F_{1} \: \cH(X) \to \cH(Y), \ F_{2} \: \cH(Y) \to \cH(X), \quad F_{1} \: A \mapsto \overline{f_{1}(A)}, \ F_{2} \: B \mapsto \overline{f_{2}(B)},$$
then $\dis(F_{1}) \le \dis(f_{1}), \ \dis(F_{2}) \le \dis(f_{2}), \ \codis(F_{1}, F_{2}) \le \codis(f_1, f_2)$.
\end{thm}

From this theorem we obtain a corollary.

\begin{thm}
For any $X, Y \in \cGH$ and every ordinal $\alpha$ we have
$$d_{GH}\bigl(\cH^{\alpha}(X), \cH^{\alpha}(Y) \bigr) \le d_{GH}(X, Y).$$
\end{thm}

\begin{proof}
Assume without loss of generality that $d_{GH}(X, Y) < \infty$, since otherwise there is nothing to prove.
Choose $\e > 0$ and consider maps $f_1 \: X \to Y, \ f_2 \: Y \to X$ such that
$\max\bigl\{\dis(f_1), \ \dis(f_2), \ \codis(f_1, f_2)\bigr\} \le 2d_{GH}(X, Y) - \e$.
We define by induction bounded maps $f^{\gamma}_{1} \: \cH^{\gamma}(X) \to \cH^{\gamma}(Y)$ and $f^{\gamma}_{2} \: \cH^{\gamma}(Y) \to \cH^{\gamma}(X)$, where $\gamma \le \alpha$.

Thus, if $\gamma = \beta + 1$, then we define for $j = 1, 2$
$$f^{\gamma}_{j} \: \cH^{\gamma}(X) \to \cH^{\gamma}(Y), \ A \mapsto \overline{f^{\beta}_{j}(A)},$$
where the right-hand side denotes the closure of the image of $A$ under the map $f^{\beta}_{j}$.

If $\gamma$ is a limit ordinal, then we define
$$f^{\gamma}_{1} \: \cH^{\gamma}(X) \to \cH^{\gamma}(Y), \ x \mapsto f^{\beta}_{1}(x), \quad \text{where the ordinal } \beta < \gamma \text{ is such that } x \in \cH^{\beta}(X),$$
$$f^{\gamma}_{2} \: \cH^{\gamma}(Y) \to \cH^{\gamma}(X), \ y \mapsto f^{\beta}_{2}(y), \quad \text{where the ordinal } \beta < \gamma \text{ is such that } y \in \cH^{\beta}(Y).$$
Thus we obtain an $\alpha$-tower of correspondences $\{R_{\beta}\}_{\beta < \alpha}, \ R_{\beta} \in \cR\bigl(\cH^{\beta}(X), \cH^{\beta}(Y)\bigr)$, where $R_{\beta} := R_{f^{\beta}_{1}, f^{\beta}_{2}}$.
Then from Lemma~\ref{lem:R_varinjlim} and Theorem~\ref{thm:cH} we obtain by induction on $\beta$ that $\dis(R_{\beta}) \le 2d_{GH}(X, Y) - \e$, whence, letting $\e \to 0$, we obtain the required result by taking $\beta = \alpha$.
\end{proof}

\begin{prop}
Let $X$ be a metric space, let $\alpha \ge \omega$ be an ordinal number.
Then $\tr_{\cH^{\alpha}(X)}(\aleph_{0}) = \diam(X)$.
\end{prop}

\begin{proof}
Assume that $X$ is bounded. The case of an unbounded $X$ is treated similarly.
Since $\tr_{\cH^{\omega}(X)}(\aleph_{0}) \le \tr_{\cH^{\alpha}(X)}(\aleph_{0}) \le \diam\bigl(\cH^{\alpha}(X)\bigr) = \diam(X)$, it is enough to prove that
$\tr_{\cH^{\alpha}(X)}(\aleph_{0}) \ge \diam(X)$ for $\alpha = \omega$.

If $X = \{pt\}$, then there is nothing to prove.
Assume that $\diam(X) \ne 0$ and, choosing $\e > 0$, consider points $x, y \in X$ such that $r := d_{X}(x, y) > \diam(X) - \e$.
Consider the family of subsets $\{Q^{(n)}\}_{n \in \N \cup \{0\}}, \ Q^{(n)} \subset \cH^{n}(X)$, given as follows. Set $Q^{(0)} = \{x, y\}$ and define $Q^{(n)} = 2^{Q^{(n - 1)}}$ for $n \ge 1$.
Also set $Q = \bigcup_{n \in \N} Q^{(n)}$, then the set $Q$ is $r$-separated and $\#Q = \aleph_{0}$. Hence, $\tr_{\cH^{\omega}(X)}(\aleph_{0}) \ge r > \diam(X) - \e$.
Letting $\e \to 0$, we obtain $\tr_{\cH^{\omega}(X)}(\aleph_{0}) \ge \diam(X)$.
\end{proof}

The following proposition is proved in the same way as the previous one.

\begin{prop}
Let $X$ be a metric space, $\alpha$ be an ordinal number.
Then $\tr_{\cH^{\alpha}(X)}(\beth_{\alpha}) \ge \tr_{X}(\aleph_{0})$.
\end{prop}

\begin{cor}\label{cor:double_tr}
Let $X$ be a metric space, $\alpha$ be an ordinal number.
Then $\tr_{\cH^{\omega + \alpha}(X)}(\beth_{\alpha}) = \diam(X)$.
\end{cor}

\begin{thm} \label{thm:dens_plus1}
Let $X$ be a metric space, $\alpha \ge \dens(X) + 1$ be an ordinal number.
Then $2d_{GH}\bigl(X, \cH^{\alpha}(X)\bigr) = \diam(X)$.
\end{thm}

\begin{proof}
Assume that $\#X \ge 2$, since otherwise there is nothing to prove.
It is enough for us to prove the inequality
$2d_{GH}\bigl(X, \cH^{\alpha}(X)\bigr) \ge \diam(X)$, since the reverse inequality holds automatically:
$$2d_{GH}\bigl(X, \cH^{\alpha}(X)\bigr) \le \max\Bigl\{ \diam(X), \ \diam\bigl(\cH^{\alpha}(X)\bigr)\Bigr\} = \diam(X).$$
Consider three cases.
\begin{enumerate}
\item
Let $X$ be a finite metric space.
Then $\dens(X) = \# X$.
Let us show that for every $n \in \N$ we have
$$\tr_{\cH^{n}(X)}(n + 1) = \diam(X) > 0.$$
We verify this by induction on $n$.
Indeed, if $n = 1$, then $\tr_{\cH^{1}(X)}(2) = \diam(X)$,
and if it is already proved for $n$, then by items~\ref{prop:properties:1} and~\ref{prop:properties:5} of Proposition~\ref{prop:properties}, as well as by Proposition~\ref{prop:H_trace} we obtain
$$\diam(X) \ge \tr_{\cH^{n + 1}(X)}(n + 2) \ge \tr_{\cH^{n + 1}(X)}\bigl(2^{n + 1} - 1\bigr) \ge \tr_{\cH^{n}(X)}(n + 1) = \diam(X),$$
whence it follows that $\tr_{\cH^{n + 1}(X)}(n + 2) = \diam(X)$.
Then for a finite $\alpha \ge \dens(X) + 1 = \#X + 1$, substituting $n = \alpha$ into the verified equality and noting that $\alpha + 1 \ge \#X + 2$, by item~\ref{prop:properties:5} of Proposition~\ref{prop:properties} we obtain
$$\tr_{\cH^{\alpha}(X)}(\# X + 2) = \diam(X).$$
For $\alpha \ge \omega$ the same equality follows from the monotonicity of the trace under inclusion of subspaces (item~\ref{prop:properties:6} of Proposition~\ref{prop:properties}) and the upper bound from item~\ref{prop:properties:1}.
On the other hand, $\tr_{X}(\# X + 2) = 0$, whence, by Theorem~\ref{thm:main}, we obtain that
$$2d_{GH}\bigl(X, \cH^{\alpha}(X)\bigr) \ge \bigl| \tr_{\cH^{\alpha}(X)}(\# X + 2) - \tr_{X}(\# X + 2) \bigr| = \diam(X).$$

\item
Let $X$ be an infinite separable space.
Then $\# X \ge \aleph_{0}, \ \dens(X) = \aleph_{0}$.
Then from the condition we obtain $\alpha \ge \omega + 1$.
Hence, by items~\ref{prop:properties:1} and~\ref{prop:properties:5} of Proposition~\ref{prop:properties} and by Corollary~\ref{cor:double_tr} we obtain that
$$\diam(X) \ge \tr_{\cH^{\alpha}(X)}(\beth_{1}) \ge \tr_{\cH^{\omega + 1}(X)}(\beth_{1}) = \diam(X),$$
whence it follows that $\tr_{\cH^{\alpha}(X)}(\beth_{1}) = \diam(X)$.

On the other hand, since $\dens(X) = \beth_{0} < \beth_{1}$, by item~\ref{prop:properties:9} of Proposition~\ref{prop:properties} we obtain that
$\tr_{X}(\beth_{1}) = 0$.
Consequently, by Theorem~\ref{thm:main} we obtain
$$2d_{GH}\bigl(X, \cH^{\alpha}(X)\bigr) \ge \bigl| \tr_{\cH^{\alpha}(X)}(\beth_{1}) - \tr_{X}(\beth_{1}) \bigr| = \diam(X).$$

\item
Let $X$ be an infinite nonseparable space.
Then $\alpha \ge \dens(X) + 1$ is an uncountable ordinal, therefore $\omega + \alpha = \alpha$ and $$\beth_{\alpha} \ge \beth_{\dens(X) + 1} = 2^{\beth_{\dens(X)}} \ge 2^{\dens(X)}.$$
Hence, by Corollary~\ref{cor:double_tr} we obtain
$$\diam(X) \ge \tr_{\cH^{\alpha}(X)}\bigl(2^{\dens(X)}\bigr) \ge \tr_{\cH^{\alpha}(X)}(\beth_{\alpha}) = \tr_{\cH^{\omega + \alpha}(X)}(\beth_{\alpha}) = \diam(X),$$
that is, $$\tr_{\cH^{\alpha}(X)}\bigl(2^{\dens(X)}\bigr) = \diam(X).$$
On the other hand, $\tr_{X}\bigl(2^{\dens(X)}\bigr) = 0$.
By Theorem~\ref{thm:main} we obtain
$$2d_{GH}\bigl(X, \cH^{\alpha}(X)\bigr) \ge \Bigl|
\tr_{\cH^{\alpha}(X)}\bigl(2^{\dens(X)}\bigr) - \tr_{X}\bigl(2^{\dens(X)}\bigr) \Bigr| = \diam(X).$$
\end{enumerate}
\end{proof}

\section{Gromov--Hausdorff distance between $l^p$-products of metric spaces}

\subsection{Estimate of traces for $l^p$-products of metric spaces}

In this section we estimate traces for $l^p$-products of metric spaces, and also express the trace of the $l^{\infty}$-product in terms of the traces of its factors.

\begin{dfn} Let $\{(X_{n}, d_{X_{n}})\}_{n=1}^{\infty}$ be a sequence of metric spaces such that $\sum_{n=1}^{\infty}(\diam(X_n))^{p} < \infty$, where $1 \le p < \infty$. Then on the Cartesian product $X = \prod_{n=1}^{\infty} X_{n}$ we introduce the metric
$$ d_{X}\bigl(\{x_n\}_{n=1}^{\infty}, \{y_n\}_{n=1}^{\infty} \bigr) = \Biggl(\sum_{n = 1}^{\infty} \bigl(d_{X_{n}}(x_{n}, y_{n}) \bigr)^{p} \Biggr)^{1/p}$$ and we call the resulting metric space $X = (l^{p})\prod_{n} X_{n}$ the $l^p$-\emph{product} of the spaces $X_n$.
\end{dfn}

\begin{dfn}
Let $\{(X_{i}, d_{X_{i}})\}_{i \in I}$ be an arbitrary family of metric spaces such that $\sup_{i \in I} \diam(X_{i})<\infty$.
Then on the Cartesian product $X = \prod_{i \in I} X_{i}$ we introduce the metric
$$d_{X}\bigl(\{x_{i}\}_{i \in I}, \{y_{i}\}_{i \in I} \bigr) = \sup_{i \in I} d_{X_{i}}(x_{i}, y_{i})$$ and we call the resulting metric space $X = (l^{\infty})\prod_{i \in I} X_{i}$ the $l^{\infty}$-\emph{product} of the spaces $X_{i}$.
\end{dfn}

Note that the diameter of $l^p$-products (where $1 \le p \le \infty$) is computed as
$$\diam\biggl((l^p)\prod_{n=1}^{\infty} X_{n}\biggr) = \Bigl\| \bigl\{ \diam(X_{n}) \bigr\}_{n \in \N} \Bigr\|_{l^{p}}.$$

\begin{thm}\label{thm:tr_prod_1}
\begin{enumerate}
\item
Let $\{X_{n}, d_{X_{n}}\}_{n = 1}^{\infty}$ be a sequence of metric spaces such that
$\sum_{n=1}^{\infty}(\diam(X_n))^{p} < \infty$, where $1 \le p < \infty$.
Set $X = (l^{p})\prod_{n} X_{n}$.
Then for every cardinal number $\vk \ge 2$ we have
$$\tr_{X}(\vk) \ge \Biggl( \sum_{n = 1}^{\infty} \bigl(\tr_{X_{n}}(\vk)\bigr)^{p} \Biggr)^{1/p}.$$
\item
Let $\{(X_{i}, d_{X_{i}})\}_{i \in I}$ be an arbitrary family of metric spaces such that $\sup_{i \in I} \diam(X_{i})<\infty$. Set $X = (l^{\infty})\prod_{i \in I} X_{i}$.
Then for every cardinal number $\vk \ge 2$ we have
$$\tr_{X}(\vk) \ge \sup_{i \in I} \tr_{X_{i}}(\vk).$$
\end{enumerate}
\end{thm}

\begin{proof}
We prove only the first item, since the second one is obtained similarly.
Choose $\e > 0$ and consider numbers $r_{n} < \tr_{X_{n}}(\vk)$, $n \in \N$, such that
$$r := \bigl\|\{r_{n}\}_{n \in \N}\bigr\|_{l^{p}} \ge \Bigl\|\bigl\{\tr_{X_{n}}(\vk)\bigr\}_{n \in \N}\Bigr\|_{l^{p}} - \e.$$
Let $Q_{n} \subset X_{n}$ be an $r_{n}$-separated set of cardinality $\# Q_{n} = \vk$ from the definition of the trace.
Let us enumerate the elements of these sets: $Q_{n} = \{q_{n, \alpha}\}_{\alpha \in A}$, where $\# A = \vk$, with $q_{n, \alpha} \ne q_{n, \beta}$ for $\alpha \ne \beta$.
Note that then the set
$$Q := \biggl\{ \{q_{n, \alpha}\}_{n \in \N} \ \Big| \ \alpha \in A \biggr\} \subset \prod_{n \in \N} Q_{n} \subset X$$
has cardinality $\# Q = \vk$. We verify that $Q$ is $r$-separated.
Indeed, consider two arbitrary distinct points $q, q' \in Q$, then for some $\alpha, \beta \in A$ we have $q = \{q_{n, \alpha}\}_{n \in \N}, \ q' = \{q_{n, \beta}\}_{n \in \N}$.
Therefore, by the $r_{n}$-separatedness of the set $Q_{n}$ we have
$$d_{X}(q, q') = \Bigl\| \bigl\{d_{X_{n}}(q_{n, \alpha}, q'_{n, \beta}) \bigr\}_{n \in \N} \Bigr\|_{l^{p}} \ge \bigl\| \{r_{n}\}_{n \in \N} \bigr\|_{l^{p}} = r.$$
Hence, $$\tr_{X}(\vk) \ge \Bigl\|\bigl\{\tr_{X_{n}}(\vk)\bigr\}_{n \in \N}\Bigr\|_{l^{p}} - \e.$$
Letting $\e \to 0$, we obtain the required result.
\end{proof}

\begin{thm}
Let $\{(X_{i}, d_{X_{i}})\}_{i \in I}$ be an arbitrary family of metric spaces such that $\sup_{i \in I} \diam(X_{i})<\infty$.
For every nonempty subset $J \subset I$ we set
$$X_{J} = (l^{\infty}) \prod_{i \in J} X_{i}.$$
Set $X = X_{I}$.
Then for every cardinal number $\vk \ge 2$ we have
$$\tr_{X}(\vk) = \sup_{\# J \le \vk^{2}} \tr_{X_{J}}(\vk),$$
where the supremum is taken over all nonempty subsets $J \subset I$ of cardinality $\# J \le \vk^{2}$.
\end{thm}

\begin{proof}
Note that for every nonempty $J \subset I$ the space $X_{J}$ is isometrically embedded in $X_{I}$.
Therefore every $r$-separated set in $X_{J}$ is an $r$-separated set in $X$.
Consequently, $\tr_{X}(\vk) \ge \tr_{X_{J}}(\vk)$, whence
$$\tr_{X}(\vk) \ge \sup_{\# J \le \vk^{2}}\tr_{X_{J}}(\vk).$$
Therefore it is enough for us to prove only the reverse inequality.
Choose $0 < r < \tr_{X}(\vk)$ and consider an $r$-separated set $Q \subset X$ of cardinality $\# Q = \vk$.
Choose $0 < \e < r$.
We introduce a map $f \: Q^{2}_{*} \to I$, where $Q^{2}_{*} = \{(x, y) \in Q^{2} \mid x \ne y\}$, given as follows.
Consider distinct points $x = \{x_{i}\}_{i \in I}, y = \{y_{i}\}_{i \in I} \in Q$. Since $d_{X}(x, y) = \sup_{i \in I} d_{X_{i}}(x_{i},y_{i}) \ge r$, there exists $i \in I$ such that $d_{X_{i}}(x_{i}, y_{i}) \ge r - \e$; we set $f(x, y) = i$.
Consider $J := f(Q)$, the image of $Q$ under the map $f$; then $\# J \le (\# Q)^{2}$.
Also consider the natural projection
$p_{J} \: X \to X_{J}$.
Then by our construction the set $p_{J}(Q) \subset X_{J}$ is $(r - \e)$-separated and $\# p_{J}(Q) = \# Q = \vk$.
Consequently, $\tr_{X_{J}}(\vk) \ge r - \e$. Letting $\e \to 0$ and $r \to \tr_{X}(\vk)$, we obtain the required result.
\end{proof}

\begin{thm}
Let $\{(X_{i}, d_{X_{i}})\}_{i \in I}$ be an arbitrary family of metric spaces such that $\sup_{i \in I} \diam(X_{i})<\infty$.
Set $X = (l^{\infty}) \prod_{i \in I} X_{i}$.
Let $\vk \ge 2$ be a cardinal number. Then
$$\tr_{X}\bigl(\vk^{\# I}\bigr) \ge \inf_{i \in I} \tr_{X_{i}}(\vk).$$
\end{thm}

\begin{proof}
Choose $\e > 0$ and consider numbers $r_{i} < \tr_{X_{i}}(\vk)$, $i \in I$, such that
$$r := \inf_{i \in I} r_{i} \ge \inf_{i \in I} \tr_{X_{i}}(\vk) - \e.$$
Let $Q_{i} \subset X_{i}$ be an $r_{i}$-separated set of cardinality $\# Q_{i} = \vk$ from the definition of the trace for all $i \in I$.
Then the set $Q = \prod_{i \in I} Q_{i} \subset X$ has cardinality $\vk^{\# I}$.
We verify that it is $r$-separated.
Indeed, consider two arbitrary distinct points $q, q' \in Q$; let $q = \{q_{i}\}_{i \in I}, \ q' = \{q'_{i}\}_{i \in I}$. Since $q \ne q'$, there exists $i^{*} \in I$ for which $q_{i^{*}} \ne q'_{i^{*}}$. Then
$$d_{X}(q, q') = \sup_{i \in I} d_{X_{i}}(q_{i}, q'_{i}) \ge d_{X_{i^{*}}}(q_{i^{*}}, q'_{i^{*}}) \ge r_{i^{*}} \ge \inf_{i \in I} r_{i} = r.$$
Hence,
$$\tr_{X}\bigl(\vk^{\# I}\bigr) \ge r \ge \inf_{i \in I} \tr_{X_{i}}(\vk) - \e.$$
Letting $\e \to 0$, we obtain the required result.
\end{proof}

\subsection{Calculation of the Gromov--Hausdorff distance for $l^{\infty}$-products using the Ramsey and Erd\H{o}s--Rado theorems}

For every cardinal number $\lambda$ we denote by $\lambda^{+}$ the cardinal number following it.
We begin with combinatorial definitions.
Let $\kappa, \lambda, \mu$ be cardinal numbers.
By the arrow relation
$$\kappa \rightarrow (\lambda )_{\mu }$$
we mean the validity of the following assertion: for any complete graph on $\kappa$ vertices, under an absolutely arbitrary coloring of its edges in (at most) $\mu$ colors, there always exists a complete monochromatic subgraph of cardinality $\lambda$.
In these terms we state the Ramsey and Erd\H{o}s--Rado theorems (see, for example,~\cite{ErdosHajnal1984, Jech2003}).

\begin{thm}[Ramsey]
For every positive integer $m$ we have
$$\aleph_{0} \rightarrow (\aleph_{0})_{m}.$$
\end{thm}

\begin{thm}[Erd\H{o}s--Rado]
For every cardinal $\lambda \ge \aleph_{0}$ we have
$$(2^{\lambda})^{+} \rightarrow (\lambda^{+})_{\lambda}.$$
\end{thm}

\begin{thm}\label{thm:trace_of_prod}
\begin{enumerate}
\item
Let $\{(X_i, d_{X_i})\}_{i \in I}$ be a family of metric spaces for which $\sup_{i \in I} \diam(X_{i}) < \infty$, with $\#I = \lambda \ge \aleph_{0}$. Set
$$X = (l^\infty)\prod_{i \in I} X_{i}.$$
Then the inequality holds\/\rom:
$$\tr_{X} \bigl( (2^\lambda)^{+} \bigr) \le \sup_{i \in I} \tr_{X_i} \bigl( \lambda^{+} \bigr). $$
\item
Let $N$ be a finite cardinal number, let $X_{1}, \dots, X_{N}$ be metric spaces. Set
$$X = (l^\infty)\prod_{n = 1}^{N} X_{n}.$$
Then the inequality holds\/\rom:
$$\tr_{X}(\aleph_{0}) = \max\bigl\{\tr_{X_{1}}(\aleph_{0}), \dots, \tr_{X_{N}}(\aleph_{0})\bigr\}.$$
\end{enumerate}
\end{thm}

\begin{proof}
Set $r = \sup_{i \in I} \tr_{X_i}(\lambda^{+})$.
We argue by contradiction. So, suppose that $\tr_{X}((2^\lambda)^{+}) > r$.
Then there exists $\e > 0$ such that in $X$ there is an $(r + \e)$-separated subset $Q$ of cardinality $\#Q = (2^\lambda)^{+}$.
We define a coloring of the complete graph on the vertices of the set $Q$ in the set of colors $I$, given as follows.
By the definition of the metric on the space $X$, for any two distinct points $x, y \in Q$ we have
$$d_{X}(x, y) = \sup_{i \in I} d_{X_i}(x_{i}, y_{i}) \ge r + \e.$$
Consequently, there exists an index $i^{*} \in I$ such that
$d_{X_{i^{*}}}(x_{i^{*}}, y_{i^{*}}) \ge r + \e/2$.
We color the edge $\{x, y\}$ in the color $i^{*}$.
By the Erd\H{o}s--Rado theorem, since $\#Q = (2^\lambda)^{+}$ and the number of colors is $\lambda$, there exists in $Q$ a complete subgraph $Q' \subset Q$ of cardinality $\# Q' = \lambda^{+}$, all of whose edges are colored in some color $j \in I$.
This means that the image $p_{j}(Q')$ of the set $Q'$ under the canonical projection $p_{j}\: X \to X_{j}$ onto $X_j$ has cardinality $\# p_{j}(Q') = \lambda^{+}$, and the set $p_{j}(Q')$ is ($r + \e/2$)-separated in $X_{j}$.
Hence, $\tr_{X_{j}}(\lambda^{+}) \ge r + \e/2$, which contradicts the definition of $r$.
\end{proof}

\begin{cor} \label{cor:tr_prod}
\begin{enumerate}
\item
Let $X_0$ be a metric space, let a set $I$ of cardinality $\#I = \lambda \ge \aleph_{0}$ be given. Set
$$X = (l^\infty)\prod_{i \in I} X_{0}.$$
Then the inequality holds\/\rom:
$$\tr_{X_{0}}\bigl( (2^\lambda)^{+} \bigr) \le \tr_{X} \bigl( (2^\lambda)^{+} \bigr) \le \tr_{X_{0}} \bigl( \lambda^{+} \bigr). $$
\item
Let $N$ be a finite cardinal number, let $X_{0}$ be a metric space. Set
$$X = (l^\infty)\prod_{n = 1}^{N} X_{0}.$$
Then the equality holds\/\rom:
$$\tr_{X}(\aleph_{0}) = \tr_{X_{0}}(\aleph_{0}).$$
\end{enumerate}
\end{cor}

\begin{proof}
\begin{enumerate}
\item
The first inequality is ensured by Theorem~\ref{thm:tr_prod_1}, and the second one by Theorem~\ref{thm:trace_of_prod}.
\item
From Theorems~\ref{thm:tr_prod_1} and~\ref{thm:trace_of_prod} we obtain that
$\tr_{X_{0}}(\aleph_{0}) \le \tr_{X}(\aleph_{0}) \le \tr_{X_{0}}(\aleph_{0})$, therefore the equality $\tr_{X}(\aleph_{0}) = \tr_{X_{0}}(\aleph_{0})$ holds.
\end{enumerate}
\end{proof}

The following theorem is known; it estimates the Gromov--Hausdorff distance between products of metric spaces.
The question arises: what are sufficient conditions for this bound to be exact?

\begin{thm}[\cite{Me_GH_Product_2026}] \label{thm:prod_2026}
Let $\{X_{i} \}_{i \in I}$, $\{Y_{i} \}_{i \in I}$ be arbitrary families of metric spaces such that $\sup_{i \in I} \diam(X_{i})<\infty$ and $\sup_{i \in I} \diam(Y_{i})<\infty$.
Then
$$d_{GH} \Biggl((l^{\infty}) \prod_{i \in I} X_{i}, \ (l^{\infty})\prod_{i \in I}Y_{i} \Biggr) \le \sup_{i \in I}d_{GH}(X_{i}, {Y_{i}}).$$
\end{thm}

It turns out that in the case where all components of the product coincide, this sufficient condition can be obtained with the help of the theorems just proved.
Thus, the following theorem holds.

\begin{thm} \label{thm:Calculation_Of_GH_between_products_by_Tr}
\begin{enumerate}
\item
Let $\lambda \ge \aleph_{0}$ be a cardinal number.
Let the metric spaces $X_0, Y_0$ be such that
$$\tr_{X_0}(\lambda^{+}) = \tr_{X_0}\bigl((2^{\lambda})^{+}\bigr) \quad \text{ and } \quad \tr_{Y_0}(\lambda^{+}) = \tr_{Y_0}\bigl((2^{\lambda})^{+}\bigr),$$
with
$$2d_{GH}(X_0, Y_0) = \bigl|\tr_{X_0}(\lambda^{+}) - \tr_{Y_0}(\lambda^{+})\bigr|.$$
Let a set $I$ of cardinality $\#I = \lambda$.
Then the equality holds\/\rom:
$$d_{GH} \Biggl((l^{\infty}) \prod_{i \in I} X_{0}, \ (l^{\infty})\prod_{i \in I}Y_{0} \Biggr) = d_{GH}(X_{0}, Y_{0}).$$
\item
Let $N$ be a finite cardinal number.
Let the metric spaces $X_0, Y_0$ be such that
$$2d_{GH}(X_0, Y_0) = \bigl|\tr_{X_0}(\aleph_{0}) - \tr_{Y_0}(\aleph_{0})\bigr|.$$
Then the equality holds\/\rom:
$$d_{GH} \Biggl((l^{\infty}) \prod_{i = 1}^{N} X_{0}, \ (l^{\infty})\prod_{i = 1}^{N} Y_{0} \Biggr) = d_{GH}(X_{0}, Y_{0}).$$
\end{enumerate}
\end{thm}

\begin{proof}
We prove only the first item, since the second item is obtained in the same way.
Set
$$X = (l^{\infty}) \prod_{i \in I} X_{0} \quad \text{ and } \quad Y = (l^{\infty}) \prod_{i \in I} Y_{0}.$$
By Theorem~\ref{thm:prod_2026} it suffices to show that $d_{GH}(X, Y) \ge d_{GH}(X_0, Y_0)$.
Thus, by Corollary~\ref{cor:tr_prod}, taking into account that $\tr_{X_0}(\lambda^{+}) = \tr_{X_0}\bigl((2^{\lambda})^{+}\bigr)$ and $\tr_{Y_0}(\lambda^{+}) = \tr_{Y_0}\bigl((2^{\lambda})^{+}\bigr)$, we obtain that
$$\tr_{X}\bigl((2^{\lambda})^{+}\bigr) = \tr_{X_0}(\lambda^{+}) \quad \text{ and } \quad \tr_{Y}\bigl((2^{\lambda})^{+}\bigr) = \tr_{Y_{0}}(\lambda^{+}).$$
Then by Theorem~\ref{thm:main} we obtain that
$$2d_{GH}(X, Y) \ge \Bigl| \tr_{X}\bigl((2^{\lambda})^{+}\bigr) - \tr_{Y}\bigl((2^{\lambda})^{+}\bigr)\Bigr| = \bigl| \tr_{X_0}(\lambda^{+}) - \tr_{Y_0}(\lambda^{+})\bigr| = 2d_{GH}(X_0, Y_0),$$
which completes the proof.
\end{proof}

\begin{examp}
We show how Theorem~\ref{thm:Calculation_Of_GH_between_products_by_Tr} can be used to compute the Gromov--Hausdorff distance.
Let $X_0$ be an arbitrary nontrivial metric compactum (for example, the segment $[0, 1]$).
Set $D = \diam(X_0) > 0$.
We construct the space $Y_0$ by adding to $X_0$ a countable ``cloud'' of discrete points.
Namely, we set $Y_0 = X_0 \sqcup Z$, where $Z = \{z_1, z_2, \dots\}$ is a countable set.
We define the metric on $Y_0$ as follows:
\begin{itemize}
\item
on the subspace $X_0$ the metric coincides with the original one;
\item
for any $i \ne j$ we set $d_{Y_0}(z_i, z_j) = D$;
\item
for any $x \in X_0$ and $z_i \in Z$ we set $d_{Y_0}(x, z_i) = D/2$.
\end{itemize}
Note that $Y_0$ is a metric space of diameter $D$.

Since the space $X_0$ is compact (and hence totally bounded), by Corollary~\ref{cor:tr_iff} we have $\tr_{X_0}(\aleph_{0}) = 0$.
On the other hand, the set $Z \subset Y_0$ is countable and $D$-separated, therefore $\tr_{Y_0}(\aleph_{0}) = D$.
Let us estimate the Gromov--Hausdorff distance $d_{GH}(X_0, Y_0)$.
From Theorem~\ref{thm:main} we immediately obtain the lower bound:
$$2d_{GH}(X_0, Y_0) \ge |\tr_{X_0}(\aleph_{0}) - \tr_{Y_0}(\aleph_{0})| = D.$$
To obtain the upper bound, we consider the following maps:
$f \: X_{0} \to Y_{0}$ --- the natural isometric embedding and
$$g \: Y_{0} \to X_{0}, \quad g(y) =
\begin{cases}
y, & \text{ if } y \in X_{0}, \\
x_0, & \text{ if } y \in Z,
\end{cases}
$$
where $x_0 \in X_0$ is an arbitrarily chosen point.
Then $\dis(f) = 0, \ \dis(g) = D, \ \codis(f, g) = D/2$.
Consequently,
$$2d_{GH}(X_0, Y_0) \le \max\bigl\{ \dis(f), \ \dis(g), \ \codis(f, g) \bigr\} = D.$$
Thus, $2d_{GH}(X_0, Y_0) = D$, which coincides exactly with the difference of the $\aleph_{0}$-traces.

Now let $N \ge 2$ be an arbitrary positive integer. Consider the $N$-dimensional Cartesian powers $X_0^N$ and $Y_0^N$ with the metric of $l^{\infty}$-products.
The space $X_0^N$ is compact (as a finite product of compact spaces), while $Y_0^N$ has a complicated combinatorial structure.
By the second item of~\ref{thm:Calculation_Of_GH_between_products_by_Tr}, we obtain
$$d_{GH}(X_0^N, Y_0^N) = d_{GH}(X_0, Y_0) = \frac{D}{2}.$$
\end{examp}

\subsection{Lower bound for the Gromov--Hausdorff distance by means of the Erd\H{o}s--Rado theorem}

\begin{thm}
Let $\lambda$ be an infinite cardinal number. Let $\{X_\alpha\}_{\alpha \in A}$ and $\{Y_\beta\}_{\beta \in B}$ be families of nonempty metric spaces satisfying the conditions\/\rom:
$$
\sup_{\alpha \in A} \diam(X_\alpha) < \infty \quad \text{and} \quad \sup_{\beta \in B} \diam(Y_\beta) < \infty,
$$
and the cardinalities of the index sets are such that $\#A \le \lambda, \, \#B \le \lambda$. Suppose that the cardinality of at least one of the $l^\infty$-products $X = (l^\infty) \prod_{\alpha \in A} X_\alpha$ or $Y = (l^\infty) \prod_{\beta \in B} Y_\beta$ is sufficiently large, namely\/\rom:
$$
\max\{\#X, \#Y\} \ge (2^\lambda)^+.
$$
Then the following lower bound for the Gromov--Hausdorff distance holds\/\rom:
$$
d_{GH}(X, Y) \ge \inf_{\alpha, \beta} \inf \Bigl\{ d_{GH}(K, L) \Bigm| K \subset X_\alpha, \, L \subset Y_\beta, \, \max\{\#K, \#L\} = \lambda^+ \Bigr\}.
$$
\end{thm}

\begin{proof}
Let $\e > 0$. Without loss of generality we assume that $d_{GH}(X, Y) < \infty$, since otherwise it makes no sense to bound infinity from below, and there is nothing to prove.

By the definition of the Gromov--Hausdorff distance as the infimum of distortions, we can choose a correspondence $R \in \cR(X,Y)$ such that
$$
\dis(R) < 2d_{GH}(X, Y) + \e.
$$
Since the correspondence $R \subset X \times Y$ projects surjectively onto both factors by definition, its cardinality is certainly not less than the cardinality of each of the spaces $X$ and $Y$. Consequently, the condition on the cardinality of $X$ or $Y$ automatically implies that the set of pairs $R$ itself has cardinality
$$
\#R \ge \max\{\#X, \#Y\} \ge (2^\lambda)^{+}.
$$

Now we apply the combinatorial machinery of Ramsey theory. We regard the set of pairs $R$ as the vertices of a complete graph. Our goal is to color its edges in a number of colors not exceeding $\lambda$ and then to use the Erd\H{o}s--Rado theorem.

We define the coloring rule. Let $(x, y)$ and $(\tilde{x}, \tilde{y})$ be two arbitrary distinct vertices of our graph (that is, two different pairs of the correspondence $R$). Each point $x, \tilde{x} \in X$ and $y, \tilde{y} \in Y$ is a sequence of coordinates: $x = \{x_\alpha\}_{\alpha \in A}$ and $y = \{y_\beta\}_{\beta \in B}$. By the definition of the metric on the $l^\infty$-product as the supremum over all coordinates, we can fix indices $\alpha_* \in A$ and $\beta_* \in B$, for which the inequalities hold:
$$
d_X(x, \tilde{x}) < d_{X_{\alpha_*}}(x_{\alpha_*}, \tilde{x}_{\alpha_*}) + \e, \quad d_{X_{\alpha_{*}}}(x_{\alpha_{*}}, \tilde{x}_{\alpha_{*}}) \ne 0,
$$
$$
d_Y(y, \tilde{y}) < d_{Y_{\beta_*}}(y_{\beta_*}, \tilde{y}_{\beta_*}) + \e, \quad d_{Y_{\beta_{*}}}(y_{\beta_{*}}, \tilde{y}_{\beta_{*}}) \ne 0
$$
Note that if $x \neq \tilde{x}$, then for at least one coordinate the distance $d_{X_{\alpha_*}}(x_{\alpha_*}, \tilde{x}_{\alpha_*}) \neq 0$ (similarly for $y$ and $\tilde{y}$). We color the edge $\{(x, y), (\tilde{x}, \tilde{y})\}$ in the color encoded by the pair of indices $(\alpha_*, \beta_*)$.

Since $\#A \le \lambda$ and $\#B \le \lambda$, the cardinality of the set of colors does not exceed $\lambda \times \lambda = \lambda$.

Since $\#R \ge (2^\lambda)^+$, while there are only $\lambda$ colors in total, we are in the setting of the Erd\H{o}s--Rado theorem. According to it, the graph $R$ contains a monochromatic clique $R_1 \subset R$ of cardinality $\#R_1 = \lambda^+$, all of whose edges are colored in the same color $(\alpha_0, \beta_0)$. Thus, for any two distinct pairs $(x, y), (\tilde{x}, \tilde{y}) \in R_1$ the inequalities hold simultaneously:
$$
0 \ne d_{X_{\alpha_0}}(x_{\alpha_0}, \tilde{x}_{\alpha_0}) \le d_X(x, \tilde{x}) < d_{X_{\alpha_0}}(x_{\alpha_0}, \tilde{x}_{\alpha_0}) + \e,
$$
$$
0 \ne d_{Y_{\beta_0}}(y_{\beta_0}, \tilde{y}_{\beta_0}) \le d_Y(y, \tilde{y}) < d_{Y_{\beta_0}}(y_{\beta_0}, \tilde{y}_{\beta_0}) + \e.
$$

Let $p_{\alpha_{0}} \: X \times Y \to X_{\alpha_{0}}$ and $p_{\beta_{0}} \: X \times Y \to Y_{\beta_{0}}$ be the canonical projections.
We now define the subspaces $K \subset X_{\alpha_{0}}$ and $L \subset Y_{\beta_{0}}$ as follows:
$$
K = p_{\alpha_0}(R_1) \subset X_{\alpha_0} \quad \text{and} \quad L = p_{\beta_0}(R_1) \subset Y_{\beta_0}.
$$
By construction, the subset $R_1 \subset K \times L$ restricted to the clique is automatically a correspondence between the spaces $K$ and $L$ (since $K$ and $L$ are exactly the complete images of the projections of $R_1$). We estimate its distortion $\dis(R_1)$ as a correspondence $R_1 \in \cR(K, L)$. For any two pairs $(x,y), (\tilde{x}, \tilde{y}) \in R_1$ by the triangle inequality we have:
\begin{multline*}
\bigl|d_{X_{\alpha_0}}(x_{\alpha_0}, \tilde{x}_{\alpha_0}) - d_{Y_{\beta_0}}(y_{\beta_0}, \tilde{y}_{\beta_0})\bigr| \le \\
\le \bigl|d_{X_{\alpha_0}}(x_{\alpha_0}, \tilde{x}_{\alpha_0}) - d_X(x, \tilde{x})\bigr| + \bigl|d_X(x, \tilde{x}) - d_Y(y, \tilde{y})\bigr|
+ \bigl|d_Y(y, \tilde{y}) - d_{Y_{\beta_0}}(y_{\beta_0}, \tilde{y}_{\beta_0})\bigr|.
\end{multline*}
The first and third terms, by the monochromaticity of our clique, are strictly less than $\e$. The second term by definition does not exceed the distortion of the correspondence $\dis(R)$. Combining these estimates, we obtain:
$$
\dis(R_1) \le \dis(R) + 2\e < 2d_{GH}(X, Y) + 3\e.
$$
On the other hand, by the definition of the Gromov--Hausdorff distance between the spaces $K$ and $L$, its doubled value does not exceed the distortion of any correspondence between them, that is, $2d_{GH}(K, L) \le \dis(R_1)$. Hence we have:
$$
2d_{GH}(K, L) < 2d_{GH}(X, Y) + 3\e.
$$
Since the subspaces $K \subset X_{\alpha_0}$ and $L \subset Y_{\beta_0}$ have cardinality exactly $\lambda^+$, we can rewrite this inequality by passing to the infimum over all pairs of subsets of fixed cardinality $\lambda^+$ and over all pairs of indices $(\alpha, \beta)$:
$$
2d_{GH}(X, Y) > 2 \inf_{\alpha, \beta} \inf \Bigl\{ d_{GH}(K, L) \Bigm| K \subset X_\alpha, \, L \subset Y_\beta, \, \#K = \lambda^+, \, \#L = \lambda^+ \Bigr\} - 3\e.
$$
Letting $\e \to 0$, we arrive at the required inequality.
\end{proof}

\begin{rk}
Note that the requirements $\sup_{\alpha \in A} \diam(X_{\alpha}) < \infty$ and $\sup_{\beta \in B} \diam(Y_{\beta}) < \infty$ in the theorem can be weakened. It is enough to require that all spaces except possibly finitely many of them possess this property.
\end{rk}

\begin{examp}
We show how the previous theorem can be used to obtain a bound for the Gromov--Hausdorff distance in a certain special case.

Consider families of nonempty bounded metric spaces $\{X_\alpha\}_{\alpha \in A}$ and $\{Y_\beta\}_{\beta \in B}$ with cardinalities of the index sets $\#A \le \lambda$ and $\#B \le \lambda$, where $\lambda$ is an infinite cardinal number.
As before, we set $X = (l^{\infty}) \prod_{\alpha \in A} X_{\alpha}, \ Y = (l^{\infty}) \prod_{\beta \in B} Y_{\beta}$.
Suppose that the cardinality condition holds for at least one of the $l^\infty$-products $X$ or $Y$: $\max\{\#X, \#Y\} \ge (2^\lambda)^+$.

Suppose that each of the spaces $Z \in \{X_\alpha\}_{\alpha \in A} \cup \{Y_\beta\}_{\beta \in B}$ satisfies the following condition: for any subset $W \subset Z$ of cardinality $\#W = \lambda^+$ the equality $\diam(W) = \diam(Z)$ holds. In particular, as such spaces one can consider $\Delta_{n}(a)$ (discrete spaces with equal pairwise distances) of sufficiently large cardinality, in which the diameter of any subset of two or more points equals the diameter of the whole space.

By the theorem proved above, for any subspaces $K \subset X_\alpha$ and $L \subset Y_\beta$, having cardinality $\max\{\#K, \#L\} = \lambda^+$, the standard lower bound via the difference of diameters holds:
$$
2d_{GH}(K, L) \ge \bigl| \diam(K) - \diam(L) \bigr| = \bigl| \diam(X_\alpha) - \diam(Y_\beta) \bigr|.
$$
Then, passing to the infimum over all such subspaces and pairs of indices $(\alpha, \beta)$, from the theorem proved above we immediately obtain the following bound:
$$
2d_{GH}(X, Y) \ge \inf_{\alpha, \beta} \bigl| \diam(X_\alpha) - \diam(Y_\beta) \bigr|.
$$
Note that the quantity on the right-hand side of the inequality is the ``naive'' distance between the sets $\{ \diam(X_{\alpha}) \mid \alpha \in A\}$ and $\{\diam(Y_{\beta}) \mid \beta \in B\}$ in $\R_{+}$.
\end{examp}

The proof of the following theorem is completely analogous to the proof of the preceding one (the Ramsey theorem is used instead of the Erd\H{o}s--Rado theorem).

\begin{thm}
Let $N$ be a positive integer. Let $\{X_1, \dots, X_N\}$ and $\{Y_1, \dots, Y_N\}$ be two finite families of nonempty bounded metric spaces, and let $X = (l^\infty) \prod_{n=1}^N X_n$ and $Y = (l^\infty) \prod_{n=1}^N Y_n$ be their Cartesian products with the $l^\infty$-metric. Suppose that at least one of the spaces $X$ or $Y$ is an infinite set \rom(that is, $\max\{\#X, \#Y\} \ge \aleph_{0}$\rom). Then the following lower bound for the Gromov--Hausdorff distance holds\/\rom:
$$
d_{GH}(X, Y) \ge \inf_{n, m} \inf \Bigl\{ d_{GH}(K, L) \Bigm| K \subset X_n, \, L \subset Y_m, \, \max\{\#K, \#L\} = \aleph_{0} \Bigr\}.
$$
\end{thm}

 \end{document}